\documentclass[oneside,english,reqno]{amsart}
\usepackage[T1]{fontenc}
\usepackage[latin9]{inputenc}
\usepackage{color}
\usepackage{babel}
\usepackage{amstext}
\usepackage{amsthm}
\usepackage{amssymb}
\usepackage[unicode=true,pdfusetitle,
 bookmarks=true,bookmarksnumbered=false,bookmarksopen=false,
 breaklinks=false,pdfborder={0 0 1},backref=false,colorlinks=false]
 {hyperref}
\hypersetup{
 colorlinks=true,citecolor=blue,linkcolor=blue,linktocpage=true}

\makeatletter
\numberwithin{equation}{section}
\numberwithin{figure}{section}
\theoremstyle{plain}
\newtheorem{thm}{\protect\theoremname}[section]
  \theoremstyle{definition}
  \newtheorem{defn}[thm]{\protect\definitionname}
  \theoremstyle{remark}
  \newtheorem{rem}[thm]{\protect\remarkname}
  \theoremstyle{plain}
  \newtheorem{prop}[thm]{\protect\propositionname}
  \theoremstyle{definition}
  \newtheorem{example}[thm]{\protect\examplename}

\renewcommand{\section}{%
\@startsection{section}{1}%
  \z@{.7\linespacing\@plus\linespacing}{.5\linespacing}%
  {\normalfont\scshape\centering\bfseries}}

\renewcommand{\subsection}{%
\@startsection{subsection}{2}%
  \z@{.5\linespacing\@plus.7\linespacing}{.5\linespacing}%
  {\normalfont\bfseries}}

\makeatother

  \providecommand{\definitionname}{Definition}
  \providecommand{\examplename}{Example}
  \providecommand{\propositionname}{Proposition}
  \providecommand{\remarkname}{Remark}
\providecommand{\theoremname}{Theorem}

\begin{document}

\title{Sequential Derivatives}

\author{Steen Pedersen }

\address{Department of Mathematics, Wright State University, Dayton OH 45435,
USA}

\email{steen.pedersen@wright.edu}

\author{Joseph P. Sjoberg}

\address{Department of Mathematics, Wright State University, Dayton OH 45435,
USA}

\email{sjoberg.3@wright.edu}
\begin{abstract}
Consider a real valued function defined, but not differentiable at
some point. We use sequences approaching the point of interest to
define and study sequential concepts of secant and cord derivatives
of the function at the point of interest. If the function is the celebrated
Weierstrass function, it follows from some of our results that the
set cord derivates at any point coincides with the extended real line. 
\end{abstract}


\maketitle

\section{Introduction}

Very basic finite difference formulas in numerical analysis approximates
the derivative $f'\left(x\right)$ using a sequence $h_{n}>0$ such
that $h_{n}\to0.$ The two basic formulas are 
\[
\frac{f\left(x+h_{n}\right)-f\left(x\right)}{h_{n}}\to f'\left(x\right)\text{ and }\frac{f\left(x+h_{n}\right)-f\left(x-h_{n}\right)}{2h_{n}}\to f'\left(x\right).
\]
The first formula is \emph{Newton's difference quotient} and determines
the slope of a secant line of the graph of $f.$ Roughly the Newton
difference quotient approximates the slope of the tangent with an
error proportional to $h_{n}.$ In Newton's difference quotient we
could also use $h_{n}<0.$ The second formula is the \emph{symmetric
difference quotient} and determines the slope of a cord of the graph
of $f.$ Roughly the symmetric difference quotient approximates the
slope of the tangent with an error proportional to $h_{n}^{2}.$ 

In this note we study the limits of the Newton difference quotients
and of the symmetric difference quotients, when the function $f$
is continuous at $x,$ but fails to have a derivative at $x.$ Let
$N_{f,x}$ be the set of limits of the Newton difference quotient
for all $h_{n}$ such that the limit exists in the extended real numbers.
And let $S_{f,x}$ be the set of limits of the difference quotient
\[
\frac{f\left(x+h_{n}\right)-f\left(x-k_{n}\right)}{h_{n}+k_{n}}
\]
for all $h_{n},k_{n}>0$ such that the limit exists in the extended
real numbers. When $k_{n}=h_{n}$ this is the symmetric difference
quotient. Among our results are (\emph{a}) $N_{f,x}$ and $S_{f,x}$
are closed subsets of the extended real numbers, (\emph{b}) any closed
subset of the extended real numbers equals $N_{f,0}$ for some $f,$
(\emph{c}) $N_{f,x}$ is a subset of $S_{f,x},$ and (\emph{d}) if
$f$ is continuous on an interval then $N_{f,x}$ and $S_{f,x}$ are
intervals. In part of Section \ref{sec:Interactions-Between-Cord}
we assume $f$ is defined on a set of the form $\left\{ 0,h_{1},-k_{1},h_{2},-k_{2},\ldots\right\} ,$
$f\left(0\right)=0,$ and we assume the Newton difference quotients
\[
\frac{f\left(h_{n}\right)}{h_{n}}\to R\text{ and }\frac{f\left(-k_{n}\right)}{-k_{n}}\to L
\]
converge to real numbers. We show that the set of limits $S$ of the
sequences 
\[
\frac{f\left(x+h_{i_{n}}\right)-f\left(x-k_{j_{n}}\right)}{h_{i_{n}}+k_{j_{n}}}
\]
obtained by considering subsequences $h_{i_{n}},k_{j_{n}}$ of the
sequences $h_{n},k_{n},$ depends on properties of the sequences $h_{n},k_{n}.$
For example, we show, (\emph{e}) if $h_{n},k_{n}$ both decay to zero
at the same polynomial rate, then $S$ is the interval with endpoints
$R$ and $L,$ (\emph{f}) if $h_{n},k_{n}$ decay at the same exponetial
rate, then $S$ is a discrete set whose only accumulation points are
$R$ and $L,$ and (\emph{g}) if $h_{n},k_{n}$ decay at the differten
exponetial rates, then whether $S$ is a discrete set or an interval
depend on the rates of decay. 

\section{Sequential Secant Derivatives\label{sec:Sequential-Secant-Derivatives}}

We consider derivatives of real valued functions, our derivatives
are defined in terms of sequences and we allow them to be infinite.
Denote the real line by $\mathbb{R}$ and the \emph{extended real
line} $\mathbb{R}\cup\left\{ \pm\infty\right\} $ by $\overline{\mathbb{R}}.$ 
\begin{defn}
Let $f$ be a real valued function defined a the subset $D$ of $\mathbb{R}$
and let $x\in D.$ We say $L$ in $\overline{\mathbb{R}}$ is a \emph{sequential
secant derivative of $f$ at $x,$} if there is a sequence $h_{n}\neq0,$
such that $h_{n}\to0,$ $x+h_{n}\in D,$ and 
\begin{equation}
Df\left(x,h_{n}\right):=\frac{f\left(x+h_{n}\right)-f\left(x\right)}{h_{n}}\underset{n\to\infty}{\longrightarrow}L.\label{eq:Def-Sequential-Derivative}
\end{equation}
We say $L$ is a \emph{right hand sequential secant derivative of
$f$ at $x,$} if $h_{n}>0,$ and a \emph{left hand sequential secant
derivative of $f$ at $x,$} if $h_{n}<0.$ We will abbreviate $h_{n}>0$
and $h_{n}\to0$ as $n\to\infty$ by writing $h_{n}\searrow0.$ 
\end{defn}

Clearly, $f$ is differentiable at $x$ with derivative $L$ if and
only if (\ref{eq:Def-Sequential-Derivative}) holds for every $h_{n}\to0$
with $x+h_{n}\in D.$ The details can be found in any beginning analysis
book, e.g., \cite{Str00} or \cite{Ped15}. 
\begin{rem}
\label{sec-1-rem:Weierstrass}The definition of sequential secant
derivative is motivated by Weierstrass' proof, see \cite{Wei86} or
\cite{Ped15}, that the Weierstrass functions 
\[
W\left(x\right):=\sum_{k=0}^{\infty}a^{n}\cos\left(b^{n}\pi x\right),
\]
where $0<a<1$ is a real number, $b$ is an odd integer and $ab>1+\frac{3\pi}{2},$
have $\pm\infty$ (in our terminology) as sequential secant derivatives
at any point $x.$ More precisely, Weierstrass showed there are sequences
$h_{n}^{\pm}\searrow0,$ such that $\left|DW\left(x,-h_{n}^{-}\right)\right|\to\infty,$
$\left|DW\left(x,h_{n}^{+}\right)\right|\to\infty,$ and $DW\left(x,-h_{n}^{-}\right)$
and $DW\left(x,h_{n}^{+}\right)$ have different signs for all sufficiently
large $n.$ 
\end{rem}

To state and prove our results we need some terminology about subsets
of the extended real numbers, this terminology is introduced in the
following definition. 
\begin{defn}
\label{def:closed-dense-isolated-interval}Let $S$ be a subset of
the extended real line $\overline{\mathbb{R}}.$ 

(\emph{a}) We say $S$ is \emph{closed,} if any $L\in\overline{\mathbb{R}}$
for which there is a sequence of real numbers $L_{n}\in S\cap\mathbb{R}$
such that $L_{n}\to L$ must be in $S.$ 

(\emph{b}) We say a set of real numbers \emph{$A$ is dense in $S,$}
if for any $L$ in $S,$ there is a sequence $a_{k}$ in $A,$ such
that $a_{k}\to L.$ 

(\emph{c}) A point $L$ in $S$ \emph{is isolated in $S,$} if no
sequence $a_{k}$ of points in $S$ with $a_{k}\neq L$ satisfies
$a_{k}\to L.$ 

(\emph{d}) A \emph{closed interval in $\overline{\mathbb{R}}$} is
a set of the form $\left[a,b\right]:=\left\{ t\in\overline{\mathbb{R}}:a\leq t\leq b\right\} ,$
where $a<b$ are in $\overline{\mathbb{R}}.$ 
\end{defn}

For a bounded set $S$ this notion of ``closed'' agrees with the
usual notion of a closed subset of the real line and for unbounded
sets $S$ it agrees with the notion of a closed subset of the two-point-compactification
of the real line. Similar remarks apply to the other terms in Definition
\ref{def:closed-dense-isolated-interval}.

We begin by showing that the set of all secant derivatives at a point
is a closed subset of the extended real numbers. Conversely, we show
that any non-empty closed subset of the extended real numbers is the
set of secant derivatives at $0$ of some function defined on the
closed interval $\left[0,1\right].$ 
\begin{thm}
\label{thm:Seq-a-closed-set}Let $f$ be a real valued function defined
a the subset $D$ of $\mathbb{R}$ and let $x\in D.$ The set of sequential
secant derivatives of $f$ at $x$ is a closed subset of $\overline{\mathbb{R}}$. 
\end{thm}

\begin{proof}
Suppose the real numbers $L_{n}$ are sequential secant derivatives
of $f$ at $x,$ $L\in\overline{\mathbb{R}},$ and $L_{n}\to L.$
We must show $L$ is a secant derivative of $f$ at $x.$ For each
$n,$ let $h_{n,m}\neq0$ be such that $h_{n.m}\to0$ as $m\to\infty$
and 
\[
\frac{f\left(x+h_{n,m}\right)-f\left(x\right)}{h_{n,m}}\underset{m\to\infty}{\longrightarrow}L_{n}.
\]
Pick $N_{n}$ such that 
\[
\left|\frac{f\left(x+h_{n,N_{n}}\right)-f\left(x\right)}{h_{n,N_{n}}}-L_{n}\right|<\frac{1}{n}.
\]
It follows that 
\[
\frac{f\left(x+h_{n,N_{n}}\right)-f\left(x\right)}{h_{n,N_{n}}}\underset{n\to\infty}{\longrightarrow}L.
\]
Hence $L$ is sequential secant derivative of $f$ at $x.$ 
\end{proof}
Similarly, the set of right hand (and the set of left hand) sequential
secant derivatives of $f$ at a point $x$ are closed subsets of $\overline{\mathbb{R}}.$ 

In the following we explore the structure of the sets of sequential
secant derivatives of functions defined on intervals. For simplicity
we state the results for right hand sequential derivatives at $0$
for functions defined on the closed interval $\left[0,1\right].$ 
\begin{thm}
\label{thm:Seq-any-closed-set}Given any non-empty closed subset $S$
of the extended real numbers. there is a real valued function $f$
defined on the closed interval $\left[0,1\right],$ such that the
set of right hand sequential secant derivatives of $f$ at $0$ equals
the set $A.$ 
\end{thm}

\begin{proof}
Let $S$ be a closed subset of the extended real numbers. There are
several cases depending on whether or not $\pm\infty$ are in $S$
or are isolated points of $S.$ We will give a proof in case for the
situation where $\infty$ is an isolated point of $S$ and either
$-\infty$ is not in $S$ or $-\infty$ is not an isolated point of
$S.$ The modifications needed for the other cases are left for the
reader. 

Let $a_{1},a_{2},\ldots$ be a countable dense subset of $S\setminus\left\{ \infty\right\} $
consisting of real numbers.  Let $f\left(0\right)=0.$ Let $\xi_{n}>0$
be a strictly decreasing sequence such that $\xi_{1}=1$ and $\xi_{n}\to0.$
For each $n,$ partition the interval $\left(\xi_{n+1},\xi_{n}\right]$
into $n+1$ subintervals $\left(\xi_{n,k-1},\xi_{n,k}\right],$ $k=1,2,\ldots,n+1.$
For $x$ in an interval of the form $\left(\xi_{n,k-1},\xi_{n,k}\right]$
with $1\leq k\leq n$ let $f\left(x\right)=a_{k}x,$ for $x$ in an
interval of the form $\left(\xi_{n,n},\xi_{n,n+1}\right]$ let $f\left(x\right)=\sqrt{x}$.
Then the graph of $f$ contains segments of the graph of the equation
$y=a_{k}x$ arbitrarily close to the origin, hence all the numbers
$a_{1},a_{2},\ldots$ are right hand sequential secant derivatives
of $f$ at $0.$ Similarly, the graph of $f$ contains segments of
the graph of $y=\sqrt{x}$ arbitrarily close to the origin, hence
$+\infty$ is a right hand sequential secant derivative of $f$ at
$0.$ 

Suppose $L$ is a right hand sequential secant derivative of $f$
at $0.$ Then there is a sequence $h_{m}\searrow0$ such that $\tfrac{f\left(h_{m}\right)}{h_{m}}\to L.$
Now each $h_{m}$ is in one of the intervals in $\left(\xi_{n,k-1},\xi_{n,k}\right].$
If $k\leq n,$ then $\tfrac{f\left(h_{m}\right)}{h_{m}}=a_{k},$ where
$k=k\left(m\right)$ depends on $m.$ If $k=n+1,$ then $\tfrac{f\left(h_{m}\right)}{h_{m}}=\frac{1}{\sqrt{h_{m}}}.$
Since $\frac{1}{\sqrt{h_{m}}}\to\infty$ and $\infty$ is isolated
in $S,$ it follows that, either $\tfrac{f\left(h_{m}\right)}{h_{m}}=\frac{1}{\sqrt{h_{m}}}$
for all but a finite number of $m$ or $\tfrac{f\left(h_{m}\right)}{h_{m}}=a_{k\left(m\right)}$
for all but a finite number of $m.$ In the first case $L=\infty\in S,$
in the second case $L\in S,$ since $a_{k\left(m\right)}\to L$ as
$m\to\infty$ and $S\setminus\left\{ \infty\right\} $ is closed. 
\end{proof}
By Theorem \ref{thm:Seq-any-closed-set} any closed set is the set
of sequential derivatives at a point of a real valued function defined
on an interval. The following theorem shows that if $f$ is continuous
on the interval the conclusion is completely different, in fact, then
the set of sequential secant derivates must be a single point or a
closed interval. 
\begin{thm}
\label{sec-1-thm:continuous}If  $f:\left[0,1\right]\to\mathbb{R}$
is continuous, then the set of right hand sequential secant derivatives
of $f$ at $0$ is either a single point or a closed interval in $\overline{\mathbb{R}}$. 
\end{thm}

\begin{proof}
Replacing $f$ by $f\left(x\right)-f\left(0\right),$ if necessary,
we may assume $f\left(0\right)=0.$ Suppose $L<M$ are right hand
derivatives of $f$ at $0.$ Let $L<K<M.$ We must show $K$ is a
right hand sequential secant derivative of $f$ at $0.$ Suppose $h_{n}\searrow0,$
$k_{n}\searrow0,$ $\tfrac{f\left(h_{n}\right)}{h_{n}}\to L$ and
$\tfrac{f\left(k_{n}\right)}{k_{n}}\to M.$ By passing to subsequences,
if necessary, we may assume 
\[
h_{1}>k_{1}>h_{2}>k_{2}>h_{3}>k_{3}>\cdots
\]
and 
\[
\frac{f\left(h_{n}\right)}{h_{n}}<K<\frac{f\left(k_{n}\right)}{k_{n}}
\]
for all $n.$ Since $\tfrac{f\left(x\right)}{x}$ is continuous on
the interval $\left[k_{n},h_{n}\right],$ it follows from the Intermediate
Value Theorem, that there are $\ell_{n},$ such that $k_{n}<\ell_{n}<h_{n}$
and $\tfrac{f\left(\ell_{n}\right)}{\ell_{n}}=K.$ This completes
the proof. 
\end{proof}
The functions constructed in the proof of Theorem \ref{thm:Seq-any-closed-set}
are not continuous. However, we have the following analog of Theorem
\ref{thm:Seq-any-closed-set} for continuous functions.
\begin{thm}
\label{sec-1-thm:sine} Any closed subinterval of the extended real
line is the set of right hand sequential secant derivatives at $0$
of some continuous function defined on the closed interval $\left[0,1\right].$
\end{thm}

\begin{proof}
If $f\left(x\right)=\sqrt{x}\sin\left(\frac{1}{x}\right)$ when $x>0$
and $f\left(0\right)=0,$ then $f$ is continuous and the collection
of all right hand secant derivatives equals the extended real line. 

Suppose $a<b$ are real numbers. If $\left|a\right|\leq\left|b\right|,$
let $f\left(0\right)=0$ and for $x>0$ let 
\[
f_{a,b}\left(x\right):=\begin{cases}
bx\sin\left(\frac{1}{x}\right) & \text{when }b\sin\left(\frac{1}{x}\right)\geq a\\
ax & \text{when }b\sin\left(\frac{1}{x}\right)<a
\end{cases}.
\]
If $\left|b\right|<\left|a\right|,$ let $f\left(0\right)=0$ and
for $x>0$ let 
\[
f_{a,b}\left(x\right):=\begin{cases}
ax\sin\left(\frac{1}{x}\right) & \text{when }b\sin\left(\frac{1}{x}\right)\geq a\\
bx & \text{when }b\sin\left(\frac{1}{x}\right)<a
\end{cases}.
\]
In either case, $f$ is continuous on the closed interval $\left[0,1\right]$
and the set of right hand sequential secant derivatives of $f$ at
the origin equals the closed interval $\left[a,b\right].$ 

The cases where one endpoint is infinite and the other endpoint is
finite is left for the reader.
\end{proof}

\section{Sequential Cord Derivatives }

In Section \ref{sec:Sequential-Secant-Derivatives} we considered
the slopes of secant lines with endpoints $\left(x,f\left(x\right)\right)$
and $\left(x+h_{n},f\left(x+h_{n}\right)\right)$ for sequences $h_{n}\neq0.$
In this section we consider the slopes of cords with endpoints $\left(x-k_{n},f\left(x-k_{n}\right)\right)$
and $\left(x+h_{n},f\left(x+h_{n}\right)\right)$ for sequences $h_{n},k_{n}>0.$ 
\begin{defn}
We say $L\in\overline{\mathbb{R}}$ is a \emph{sequential cord derivative
of $f$ at $x,$} if there is are sequences $h_{n}\searrow0,$ and
$k_{n}\searrow0,$ such that $x+h_{n}\in D,$ $x-k_{n}\in D,$ and
\[
\frac{f\left(x+h_{n}\right)-f\left(x-k_{n}\right)}{h_{n}+k_{n}}\underset{n\to\infty}{\longrightarrow}L.
\]
For the sake of brevity, we will often say \emph{cord derivative}
in place of \emph{sequential cord derivative}.
\end{defn}

We begin by showing that, if $f$ is differentiable at $x,$ then
all the sequential cord derivatives of $f$ at $x$ exist and are
equal to the derivative $f'\left(x\right)$ of $f$ at the point $x.$ 
\begin{prop}
If $f'\left(x\right)$ exists, $h_{n}\searrow0,$ and $k_{n}\searrow0,$
then 
\[
\frac{f\left(x+h_{n}\right)-f\left(x-k_{n}\right)}{h_{n}+k_{n}}\underset{n\to\infty}{\longrightarrow}f'\left(x\right).
\]
\end{prop}

\begin{proof}
\textcolor{black}{Suppose $f'\left(x\right)$ is a real number. Let
$\varepsilon>0$ be fixed. Pick $n$ such that $\frac{f\left(x+h_{n}\right)-f\left(x\right)}{h_{n}}$
and $\frac{f\left(x-k_{n}\right)-f\left(x\right)}{-k_{n}}$ both are
within $\varepsilon$ of $f'\left(x\right).$ Since }
\begin{align*}
 & \frac{f\left(x+h_{n}\right)-f\left(x-k_{n}\right)}{h_{n}+k_{n}}\\
 & =\frac{h_{n}}{h_{n}+k_{n}}\frac{f\left(x+h_{n}\right)-f\left(x\right)}{h_{n}}+\frac{k_{n}}{h_{n}+k_{n}}\frac{f\left(x-k_{n}\right)-f\left(x\right)}{-k_{n}}.
\end{align*}
Combining this with $\frac{h_{n}}{h_{n}+k_{n}}\geq0,$ $\frac{k_{n}}{h_{n}+k_{n}}\geq0,$
and $\frac{h_{n}}{h_{n}+k_{n}}+\frac{k_{n}}{h_{n}+k_{n}}=1$ we see
that 
\begin{align*}
 & \left|\frac{f\left(x+h_{n}\right)-f\left(x-k_{n}\right)}{h_{n}+k_{n}}-f'\left(x\right)\right|\\
 & \leq\frac{h_{n}}{h_{n}+k_{n}}\left|\frac{f\left(x+h_{n}\right)-f\left(x\right)}{h_{n}}-f'\left(x\right)\right|+\frac{k_{n}}{h_{n}+k_{n}}\left|\frac{f\left(x-k_{n}\right)-f\left(x\right)}{-k_{n}}-f'\left(x\right)\right|\\
 & <\frac{h_{n}}{h_{n}+k_{n}}\varepsilon+\frac{k_{n}}{h_{n}+k_{n}}\varepsilon\\
 & =\varepsilon.
\end{align*}
The cases where $f'\left(x\right)=\pm\infty$ are left for the reader.
This completes the proof. 
\end{proof}
The following provides a converse to the previous result, when $f$
is assumed to be continuous at the point of interest. 
\begin{thm}
If $f$ is continuous at $x$ and all the cord derivatives of $f$
at $x$ exists and equals $L\in\overline{\mathbb{R}},$ then $f'\left(x\right)$
exists and equals $L.$ 
\end{thm}

\begin{proof}
By assumption 
\[
\frac{f\left(x+h_{n}\right)-f\left(x-k_{n}\right)}{h_{n}+k_{n}}\underset{n\to\infty}{\longrightarrow}L
\]
for all $h_{n}\searrow0,$ and $k_{n}\searrow0.$ Suppose $h_{n}\searrow0.$
Since $f$ is continuous at $x$ we can pick $k_{n}\searrow0$ such
that $k_{n}\leq h_{n}^{2}$ and $\left|f\left(x-k_{n}\right)-f\left(x\right)\right|\leq h_{n}^{2}.$
Using 
\begin{align*}
\frac{f\left(x+h_{n}\right)-f\left(x\right)}{h_{n}} & =\frac{h_{n}+k_{n}}{h_{n}}\cdot\frac{f\left(x+h_{n}\right)-f\left(x-k_{n}\right)}{h_{n}+k_{n}}+\frac{f\left(x-k_{n}\right)-f\left(x\right)}{h_{n}}
\end{align*}
we conclude, $\frac{f\left(x+h_{n}\right)-f\left(x\right)}{h_{n}}\to L.$
Hence, the right hand derivative of $f$ at $x$ exists and equals
$L.$ Similarly, the left hand derivative of $f$ at $x$ exists and
equals $L.$ 
\end{proof}
Our next result shows that the set of sequential cord derivatives
at a point is a closed. It is the analog of Theorem \ref{thm:Seq-a-closed-set}
for cord derivatives.
\begin{thm}
\label{thm:Cord-derivatives-closed}The set of sequential cord derivatives
at $x$ is a closed subset of $\overline{\mathbb{R}}.$ 
\end{thm}

\begin{proof}
Suppose $L_{n}\in\mathbb{R}$ is a sequence of cord derivatives at
$x$ and $L_{n}\to L$ as $n\to\infty.$ We must show $L$ is a cord
derivative. For each $n,$ let $h_{n,m}\searrow0$ and $k_{n,m}\searrow0$
be such that 
\[
\frac{f\left(x+h_{n,m}\right)-f\left(x-k_{n,m}\right)}{h_{n,m}+k_{n,m}}\underset{m\to\infty}{\longrightarrow}L_{n}.
\]
Pick $N_{n}$ such that 
\[
\left|\frac{f\left(x+h_{n,N_{n}}\right)-f\left(x-k_{n,N_{n}}\right)}{h_{n,N_{n}}+k_{n,N_{n}}}-L_{n}\right|<\frac{1}{n}
\]
then 
\[
\frac{f\left(x+h_{n,N_{n}}\right)-f\left(x-k_{n,N_{n}}\right)}{h_{n,N_{n}}+k_{n,N_{n}}}\underset{n\to\infty}{\longrightarrow}L.
\]
This completes the proof. 
\end{proof}
In Theorem \ref{sec-1-thm:continuous} we showed that the set of one-sided
secant derivatives of a continuous function is a closed interval.
Our next result establishes an appropriate version of this for cord
derivatives. 
\begin{thm}
\label{sec-2-thm:f-is-continuous}If $f:\left[-1,1\right]\to\mathbb{R}$
is continuous, then the set of sequential cord derivatives of $f$
at $0$ is either the empty set, a single point, or a closed interval
in $\overline{\mathbb{R}}$. 
\end{thm}

\begin{proof}
Replacing $f$ by $f\left(x\right)-f\left(0\right),$ if necessary,
we may assume $f\left(0\right)=0.$ Suppose $L^{-}<L^{+}$ are cord
derivatives of $f$ at $0.$ Let $L^{-}<K<L^{+}.$ We must show $K$
is a sequential cord derivative of $f$ at $0.$ Suppose $h_{n}^{\pm}\searrow0,$
$k_{n}^{\pm}\searrow0,$ $\tfrac{f\left(h_{n}^{\alpha}\right)-f\left(-k_{n}^{\alpha}\right)}{h_{n}^{\alpha}+k_{n}^{\alpha}}\to L^{\alpha},$
for $\alpha=\pm.$ Hence for sufficiently large $n$
\[
\tfrac{f\left(h_{n}^{-}\right)-f\left(-k_{n}^{-}\right)}{h_{n}^{-}+k_{n}^{-}}<K<\tfrac{f\left(h_{n}^{+}\right)-f\left(-k_{n}^{+}\right)}{h_{n}^{+}+k_{n}^{+}}.
\]
Let 
\[
\phi_{n}\left(t\right)=\tfrac{f\left(th_{n}^{-}+\left(1-t\right)h_{n}^{+}\right)-f\left(-tk_{n}^{-}-\left(1-t\right)k_{n}^{+}\right)}{th_{n}^{-}+\left(1-t\right)h_{n}^{+}+tk_{n}^{-}+\left(1-t\right)k_{n}^{+}}.
\]
The $\phi_{n}$ is continuous on $\left[0,1\right],$ $\phi_{n}\left(0\right)=\tfrac{f\left(h_{n}^{+}\right)-f\left(-k_{n}^{+}\right)}{h_{n}^{+}+k_{n}^{+}}$
and $\phi_{n}\left(1\right)=\tfrac{f\left(h_{n}^{-}\right)-f\left(-k_{n}^{-}\right)}{h_{n}^{-}+k_{n}^{-}}.$
It follows from the Intermediate Value Theorem, that $\phi_{n}\left(t_{n}\right)=K$
for some $t_{n}$ between $0$ and $1.$ Setting $h_{n}'=t_{n}h_{n}^{-}+\left(1-t_{n}\right)h_{n}^{+}$
and $k_{n}'=tk_{nn}^{-}+\left(1-t_{n}\right)k_{n}^{+}$ it follows
that $h_{n}'\searrow0,$$k_{n}'\searrow0$ and $\frac{f\left(x+h_{n}\right)-f\left(x-k_{n}\right)}{h_{n}+k_{n}}=K$
for all sufficiently large $n.$ This completes the proof. 
\end{proof}

\section{Interactions Between Cord and Secant Derivatives\label{sec:Interactions-Between-Cord}}

In this section we establish some relationships between the cord and
secant derivatives. Our first result gives a condition under which
the set of secant derivatives at a point $x$ is a subset of the set
of cord derivatives at $x.$ We show that it may be a proper subset
and apply the inclusion to the Weierstrass function. 
\begin{thm}
\label{sec-3-THM:secant-subset-cord}If $f$ is continuous at $x,$
then the set of secant derivatives of $f$ at $x$ is a subset of
the set of cord derivatives of $f$ at $x$.
\end{thm}

\begin{proof}
Let $L$ be a right hand secant derivative of $f.$ By assumption
\[
\frac{f\left(x+h_{n}\right)-f\left(x\right)}{h_{n}}\underset{n\to\infty}{\longrightarrow}L
\]
for some $h_{n}\searrow0.$ Since $f$ is continuous at $x$ we can
pick $k_{n}\searrow0$ such that $k_{n}\leq h_{n}^{2}$ and $\left|f\left(x-k_{n}\right)-f\left(x\right)\right|\leq h_{n}^{2}.$
By the choice of $k_{n}$ we have $\frac{h_{n}+k_{n}}{h_{n}}\to1$
and $\frac{f\left(x-k_{n}\right)-f\left(x\right)}{h_{n}}\to0.$ Hence,
using 
\begin{align*}
\frac{f\left(x+h_{n}\right)-f\left(x\right)}{h_{n}} & =\frac{h_{n}+k_{n}}{h_{n}}\cdot\frac{f\left(x+h_{n}\right)-f\left(x-k_{n}\right)}{h_{n}+k_{n}}+\frac{f\left(x-k_{n}\right)-f\left(x\right)}{h_{n}}
\end{align*}
we conclude, $\frac{f\left(x+h_{n}\right)-f\left(x-k_{n}\right)}{h_{n}+k_{n}}\to L.$
Hence, a cord derivative of $f$ at $x$ exists and equals $L.$ 

The case where $L$ is a left hand secant derivative of $f$ is similar.
\end{proof}
\begin{example}
The Weierstrass function revisited. Weierstrass showed that $\pm\infty$
are secant derivatives of $W$ at any point $x.$ By Theorem \ref{sec-3-THM:secant-subset-cord}
$\pm\infty$ are also cord derivatives of $W$ at any point $x.$
Since $W$ is continuous on the real line, it follows from Theorem
\ref{sec-2-thm:f-is-continuous} that at any point $x,$ the set of
cord derivatives of $W$ equals the extended real line $\overline{\mathbb{R}}.$
\end{example}

In light of Theorem \ref{sec-3-THM:secant-subset-cord} a natural
question is: Can the set of secant derivatives be a proper subset
of the set of cord derivates? By considering simple examples it is
easy to see that the answer is yes. A simple example is provided by
considering $f\left(x\right)=\left|x\right|$ another example is provided
in Example \ref{sec-3-exa:intervals}. 
\begin{example}
If $f\left(x\right):=\left|x\right|$ and $-1\leq L\leq1,$ then there
exists $h_{n}\searrow0,$ and $k_{n}\searrow0,$ such that 
\[
\frac{f\left(h_{n}\right)-f\left(-k_{n}\right)}{h_{n}+k_{n}}\underset{n\to\infty}{\longrightarrow}L.
\]
\end{example}

\begin{proof}
This follows from Theorem \ref{sec-3-thm:poly}. We provide a simple
direct proof, the proof introduces some ideas used below. 

(\emph{a}) If $L=0,$ let $h_{n}=k_{n}=\frac{1}{n}.$ 

(\emph{b}) Suppose $0<L<1.$ Consider $h_{n}=\frac{1}{n}$ and $k_{n}=\frac{b}{n},$
then 
\[
\frac{\frac{1}{n}-\frac{b}{n}}{\frac{1}{n}+\frac{b}{n}}=\frac{1-b}{1+b}=L.
\]
Solving for $b$ we see  $b=\left(1-b\right)/\left(1+L\right)$ does
the job. 

(\emph{c}) If $L=1$ setting $h_{n}=\frac{1}{n}$ and $k_{n}=\frac{1}{n^{2}}$
does the job, since 
\[
\frac{\frac{1}{n}-\frac{1}{n^{2}}}{\frac{1}{n}+\frac{1}{n^{2}}}=\frac{1-\frac{1}{n}}{1+\frac{1}{n}}\to1.
\]
This also follows from Proposition \ref{sec-3-THM:secant-subset-cord}.

(\emph{d}) We leave the cases $-1\leq L<0$ for the reader. 
\end{proof}

Below we calculate the set of cord derivatives assuming the right
hand and left hand secant derivatives exists. To simplify the notation
we assume the point of interest is $x=0$ and $f\left(0\right)=0.$
We can always arrange this by considering $g\left(t\right)=f\left(t+x\right)-f\left(x\right)$
in place of $f.$ Suppose $h_{n}\searrow0$ and $k_{n}\searrow0.$
The basis for our calculations is the formula 
\begin{equation}
\frac{f\left(h_{n}\right)-f\left(-k_{n}\right)}{h_{n}+k_{n}}=\frac{h_{n}}{h_{n}+k_{n}}\cdot\frac{f\left(h_{n}\right)}{h_{n}}+\frac{k_{n}}{h_{n}+k_{n}}\cdot\frac{f\left(-k_{n}\right)}{-k_{n}}.\label{sec-3-eq:Basic-Calculation}
\end{equation}
Suppose 
\begin{equation}
\frac{h_{i_{n}}}{h_{i_{n}}+k_{i_{n}}}\to r,\text{}\frac{f\left(h_{n}\right)}{h_{n}}\to R\text{ and }\frac{f\left(-k_{n}\right)}{-k_{n}}\to L\text{ as }n\to\infty\label{sec-3-eq:Basic-Assumptions}
\end{equation}
where $R$ and $L$ are real numbers and $h_{i_{n}}$ is a subsequence
of $h_{n}$ and $k_{j_{n}}$ is a subsequence of $k_{n},$ then $0\leq r\leq1$
and 
\begin{equation}
\frac{f\left(h_{i_{n}}\right)-f\left(-k_{j_{n}}\right)}{h_{i_{n}}+k_{j_{n}}}\to rR+\left(1-r\right)L\text{ as }n\to\infty.\label{sec-3-eq:Basic-Conclusion}
\end{equation}
Where Equation (\ref{sec-3-eq:Basic-Conclusion}) follows from (\ref{sec-3-eq:Basic-Assumptions})
by replacing the sequences $h_{n}$ and $k_{n}$ by the appropriate
subsequences in (\ref{sec-3-eq:Basic-Calculation}). 

In particular, 
\begin{prop}
\label{sec-3-prop:Cord-between-Secants}If $h_{n},k_{n}\searrow0,$
$f$ is defined on the set $\left\{ 0,h_{1},-k_{1},h_{2},-k_{2},\ldots\right\} ,$
$L,R$ are real numbers, and
\[
\frac{f\left(h_{n}\right)}{h_{n}}\to R\text{ and }\frac{f\left(-k_{n}\right)}{-k_{n}}\to L,
\]
then any cord derivative of $f$ at $0$ is a real number between
$R$ and $L.$ 
\end{prop}

Below we explore the converse of this statement. When $h_{n}$ and
$k_{n}$ decay at the same polynomial rate, then any real number between
$R$ and $L$ is a cord derivative. When $h_{n}$ and $k_{n}$ decay
at the same exponential rate, then the only accumulation points of
the set of cord derivatives are $R$ and $L,$ in particular, the
set of cord derivatives is not an interval.
\begin{thm}
\label{sec-3-thm:poly}Let $a,b,m>0$ be real numbers. Suppose $p\left(n\right),q\left(n\right)$
are increasing functions and 
\[
\frac{p\left(n\right)}{n^{m}}\to a\text{ and }\frac{q\left(n\right)}{n^{m}}\to b.
\]
Let $h_{n}:=\tfrac{1}{p\left(n\right)}$ and $k_{n}:=\tfrac{1}{q\left(n\right)}.$
Let $f$ be a function defined on $\left\{ 0,h_{1},-k_{1},h_{2},-k_{2},\ldots\right\} $
and let $R,L$ be real numbers. If $f\left(0\right)=0$ and 
\[
\frac{f\left(h_{n}\right)}{h_{n}}\to R\text{ and }\frac{f\left(-k_{n}\right)}{-k_{n}}\to L
\]
then every real number between $L$ and $R$ is a cord derivative
of $f$ at $0.$ 
\end{thm}

\begin{proof}
By Equation (\ref{sec-3-eq:Basic-Conclusion}) and Theorem \ref{thm:Cord-derivatives-closed}
it is sufficient to show that given any $0<r<1$ we can find subsequences
$h_{i_{n}}$ and $k_{j_{n}}$ such that 
\[
\frac{h_{i_{n}}}{h_{i_{n}}+k_{j_{n}}}\to r.
\]
For integers $i,j$ we have 
\[
\frac{p\left(in\right)}{n^{m}}=i^{m}\frac{p\left(in\right)}{\left(in\right)^{m}}\underset{n\to\infty}{\longrightarrow}i^{m}a\text{ and }\frac{q\left(jn\right)}{n^{m}}\underset{n\to\infty}{\longrightarrow}j^{m}b.
\]
Hence
\[
\frac{h_{in}}{h_{in}+k_{jn}}=\frac{q\left(jn\right)}{q\left(jn\right)+p\left(in\right)}\underset{n\to\infty}{\longrightarrow}\frac{j^{m}b}{j^{m}b+i^{m}a}.
\]
It remains to show we can pick $i,j$ arbitrarily large such that
$\frac{j^{m}b}{j^{m}b+i^{m}a}$ is within $\varepsilon>0$ of $r.$
To this end, let $j'$ be so large that $\tfrac{1}{j'^{m}}<\varepsilon$
and $r<\frac{j'^{m}b}{j'^{m}b+a}.$ Let $i$ be such that $\frac{j'^{m}b}{j'^{m}b+i^{m}a}<r$
and let $j>j'$ be such that 
\[
\frac{j^{m}b}{j^{m}b+i^{m}a}<r\leq\frac{\left(j+1\right)^{m}b}{\left(j+1\right)^{m}b+i^{m}a}.
\]
We complete the proof by showing 
\[
\frac{\left(j+1\right)^{m}b}{\left(j+1\right)^{m}b+i^{m}a}-\frac{j^{m}b}{j^{m}b+i^{m}a}<\varepsilon.
\]
Let 
\[
f\left(t\right)=\frac{bt}{bt+i^{m}a},
\]
then we must show $f\left(j+1\right)-f\left(j\right)<\varepsilon.$
Now
\[
0<f'\left(t\right)=\frac{bai^{m}}{\left(bt+i^{m}a\right)^{2}}<\frac{bai^{m}}{2bti^{m}a}=\frac{1}{2t}
\]
uniformly in $i.$ By the Mean Value Theorem
\[
f\left(j+1\right)-f\left(j\right)=f'\left(c\right)<\frac{1}{2c}<\frac{1}{2j}<\frac{1}{2j'}<\varepsilon.
\]
This completes the proof. 
\end{proof}
\begin{example}
\label{sec-3-exa:intervals}Let $a<b$ and $c<d$ be real numbers.
Let $f_{a,b}$ be as in the proof of Theorem \ref{sec-1-thm:sine}.
The the right hand secant derivatives of $f_{a,b}$ at $0$ equals
the interval $\left[a,b\right]$ and the set of left hand secant derivatives
of $x\to f_{-d,-c}\left(-x\right)=f_{c,d}\left(x\right)$ at $0$
equals the interval $\left[c,d\right].$ Let 
\[
g\left(x\right):=\begin{cases}
f_{a,b}\left(x\right) & \text{when }x\geq0\\
f_{-d,-c}\left(-x\right) & \text{when }x<0
\end{cases}.
\]
We claim that the set of cord derivatives of $g$ is the the convex
hull $\left[\min\left\{ a,c\right\} ,\max\left\{ b,d\right\} \right]$
of the intervals $\left[a,b\right]$ and $\left[c,d\right].$ 
\end{example}

\begin{proof}
If one of $\left[a,b\right]$ and $\left[c,d\right]$ is a subinterval,
the claim follows from Theorem \ref{sec-3-THM:secant-subset-cord}
and Proposition \ref{sec-3-prop:Cord-between-Secants}. 

Since $\sin\left(\frac{1}{x}\right)=y$ has solutions $x=\frac{1}{\arcsin\left(y\right)+2\pi k}$
where $k$ is an integer, there are harmonic progressions $\alpha_{n}=\frac{1}{p+qn}$
and $\beta_{n}=\frac{1}{r+sn}$ such that $f_{a,b}\left(\alpha_{n}\right)=a\alpha_{n}$
and $f_{a,b}\left(\beta_{n}\right)=b\beta_{n}$ and similarly for
$f_{-d,-c}.$ Consequently, the claim follows from Theorem \ref{sec-3-thm:poly}.
\end{proof}

If follows from Theorem \ref{sec-3-thm:poly} that if $h_{n}$ and
$k_{n}$ decay at the same polynomial rate, then the set of cord derivatives
is an interval. It follows from our next result that, if the sequences
$h_{n}$ and $k_{n}$ decay exponentially, then the set of cord derivatives
need not be an interval. 
\begin{thm}
Suppose $a,b>1$ are real numbers. Let $h_{n}:=\tfrac{1}{a^{n}}$
and $k_{n}:=\tfrac{1}{b^{n}}.$ Let $f$ be a function defined on
$\left\{ 0,h_{1},-k_{1},h_{2},-k_{2},\ldots\right\} $ and let $R,L$
be real numbers. Assume $f\left(0\right)=0$ and 
\[
\frac{f\left(h_{n}\right)}{h_{n}}\to R\text{ and }\frac{f\left(-k_{n}\right)}{-k_{n}}\to L.
\]
\begin{itemize}
\item If $\frac{\log\left(a\right)}{\log\left(b\right)}$ is a rational
number, then $R$ and $L$ are the only accumulation points of the
set of cord derivatives of $f$ at $0.$ 
\item If $\frac{\log a}{\log b}$ is an irrational number, then every real
number between $L$ and $R$ is a cord derivative of $f$ at $0.$
\end{itemize}
\end{thm}

\begin{proof}
Note
\[
\frac{h_{i_{n}}}{h_{i_{n}}+k_{i_{n}}}=\frac{b^{j_{n}}}{b^{j_{n}}+a^{i_{n}}}=\frac{1}{1+\frac{a^{i_{n}}}{b^{j_{n}}}}.
\]
Let $s:=\frac{1-r}{r}.$ Then 
\begin{equation}
\frac{h_{i_{n}}}{h_{i_{n}}+k_{i_{n}}}\to r\quad\text{iff}\quad\frac{a^{i_{n}}}{b^{j_{n}}}\to s\quad\text{iff}\quad i_{n}\frac{\log\left(a\right)}{\log\left(b\right)}-j_{n}\to\frac{\log\left(s\right)}{\log\left(b\right)}.\label{sec-3-eq:r-to-s}
\end{equation}
If $\frac{\log\left(a\right)}{\log\left(b\right)}$ is an irrational
number, then the set 
\begin{equation}
\left\{ i\frac{\log\left(a\right)}{\log\left(b\right)}-j:i,j\in\mathbb{N}\right\} \label{sec-3-eq:rotations}
\end{equation}
is dense in the real line. 

On the other hand, if $\frac{\log\left(a\right)}{\log\left(b\right)}=\frac{p}{q}$
is a rational number, then the set in Equation (\ref{sec-3-eq:rotations})
is a subset of the fractions with denominator $q.$ Using (\ref{sec-3-eq:r-to-s})
and that the set of cord derivatives is a closed set (by Theorem \ref{thm:Cord-derivatives-closed})
the result follows. 
\end{proof}
The density of the set in Equation (\ref{sec-3-eq:rotations}) was
first proved by Nicole Oresme around $1360$ in his paper \emph{De
commensurabilitate vel incommensurabilitate motuum celi}. For an English
translation of Oresme's proof see \cite{Gra61}. A detailed analysis
of Oresme's proof is in \cite{Pla93}. More contemporary proofs and
additional historical discussion can be found in \cite{Pet83}.

\bibliographystyle{amsalpha}
\bibliography{SequentialDerivatives}

\providecommand{\bysame}{\leavevmode\hbox to3em{\hrulefill}\thinspace}
\providecommand{\MR}{\relax\ifhmode\unskip\space\fi MR }
\providecommand{\MRhref}[2]{%
  \href{http://www.ams.org/mathscinet-getitem?mr=#1}{#2}
}
\providecommand{\href}[2]{#2}
\begin{thebibliography}{Wei86}

\bibitem[Gra61]{Gra61}
Edward Grant, \emph{Nicole {O}resme and the commensurability or
  incommensurability of celestial motions}, Archive for History of Exact
  Sciences \textbf{1} (1961), no.~4, 420--458.

\bibitem[Ped15]{Ped15}
Steen Pedersen, \emph{From {C}alculus to {A}nalysis}, Springer Publishing,
  Switzerland, 2015.

\bibitem[Pet83]{Pet83}
Karl~E. Petersen, \emph{Ergodic {T}heory}, Cambridge University Press,
  Cambridge, 1983.

\bibitem[Str00]{Str00}
Robert~S. Strichartz, \emph{The {W}ay of {A}nalysis}, revised edition ed.,
  Jones \& Bartlett Learning, 2000.

\bibitem[vP93]{Pla93}
Jan von Plato, \emph{Oresme's proof of the density of rotations of a circle
  through an irrational angle}, Historia Mthematica \textbf{20} (1993),
  428--433.

\bibitem[Wei86]{Wei86}
Karl Weierstrass, \emph{Abhandlungen aus der {F}unktionenlehre}, Julius
  Springer, Berlin, 1886.

\end{thebibliography}

\end{document}